\documentclass[11pt,reqno,letterpaper]{amsart}
\usepackage{amsfonts}
\usepackage{amsmath}
\usepackage{amssymb, esint}
\usepackage{amsthm,color}
\usepackage{enumerate}
\usepackage{mathtools}
\usepackage{enumitem}
\usepackage{url}
\usepackage[hidelinks, bookmarksdepth=3]{hyperref}

\usepackage{xcolor}
\newcommand{\Part}{\operatorname{Part}}
\usepackage{cancel}
\usepackage{ulem}
\usepackage{caption,subcaption}
\usepackage{tikz}
\usepackage{tikz-3dplot}
\usepackage[T1]{fontenc}
\usepackage{mathrsfs}
\usepackage{commath}

\usepackage{longtable}

\usepackage{soul}
\usepackage[
style=alphabetic,
backend=biber,
maxbibnames=99,
minbibnames=99,
maxcitenames=99,
mincitenames=99
]{biblatex}
\AtBeginDocument{
\setlength{\abovedisplayskip}{9pt plus 2pt minus 3pt}
\setlength{\belowdisplayskip}{9pt plus 2pt minus 3pt}
\setlength{\abovedisplayshortskip}{6pt plus 2pt minus 2pt}
\setlength{\belowdisplayshortskip}{6pt plus 2pt minus 2pt}
\setlength{\jot}{4pt}
}

\DeclareFieldFormat[article]{journaltitle}{#1}
\DeclareFieldFormat[book]{title}{#1}

\tdplotsetmaincoords{70}{110}

\tdplotsetmaincoords{70}{110}

\usepackage[letterpaper,margin=1in]{geometry}

\newcommand{\N}{{\mathbb N}}

\newcommand{\cA}{{\mathcal A}}
\newcommand{\cB}{{\mathcal B}}

\newcommand{\cH}{{\mathcal H}}
\newcommand{\cI}{{\mathcal I}}
\newcommand{\cP}{{\mathcal P}}
\newcommand{\cR}{{\mathcal R}}
\newcommand{\cF}{{\mathcal F}}

\def\0{{\mathbf 0}}

\newcommand{\dist}{\operatorname{dist}}

\theoremstyle{plain}
\newtheorem{thm}{Theorem}[section]

\newtheorem{cor}[thm]{Corollary}
\newtheorem{lem}[thm]{Lemma}

\newcommand{\thistheoremnames}{}
\newtheorem*{genericthms}{\thistheoremnames}
\newenvironment{para*}[1]
{\renewcommand{\thistheoremnames}{#1}%
\begin{genericthms}}
{\end{genericthms}}

\newtheorem{defn}[thm]{Definition}

\theoremstyle{remark}

\newtheorem*{claim*}{Claim}

\newtheorem{rem}[thm]{Remark}

\begin{document}

\numberwithin{equation}{section}

\title[A quantitative container characterization]{A quantitative container characterization \\of one-sided testability}

\author[G. Carenini]{Gaia Carenini}
\address{Trinity College Cambridge, Department of Pure Mathematics and Mathematical Statistics, Centre for Mathematical Sciences, Wilberforce Road, Cambridge CB3 0WA, United Kingdom.}
\email{gc645@cam.ac.uk}

\author[C. Seth]{Cameron Seth}
\address{Cheriton School of Computer Science, University of Waterloo, 200 University Ave W, Waterloo, ON N2L 5Z5, Canada}
\email{cjmpseth@uwaterloo.ca}

\author[Y. Yoshida]{Yuichi Yoshida}
\address{National Institute of Informatics (NII), 2-1-2 Hitotsubashi, Chiyoda-ku, Tokyo 101-8430, Japan.}
\email{yyoshida@nii.ac.jp}

\begin{abstract}
We give a quantitative combinatorial characterization of size-oblivious one-sided testability in the dense graph model, resolving a question of Alon, Fischer, Newman, and Shapira. For hereditary graph properties, we prove that one-sided testability is quantitatively equivalent to the existence of suitable hypergraph containers, a central and widely used tool in modern combinatorics. Combining this equivalence with the Alon--Shapira notion of semi-hereditariness yields a quantitative characterization of arbitrary graph properties. The correspondence is effective in both directions and provides explicit translations between tester complexity and container parameters. Our proof is regularity-free and extends uniformly to every fixed finite relational signature of bounded arity, including digraphs, coloured graphs, and hypergraphs. As applications, we obtain quantitative closure results for partition properties and testers for properties defined by the existence of a linearly large induced substructure.

\end{abstract}

\maketitle

\section{Introduction}

\subsection{The dense graph model}

Throughout this paper, a \textit{graph} is a finite simple labelled graph.  Thus
self-loops and multiple edges are not allowed.  A \textit{graph property} is a
family of graphs closed under isomorphism.  A graph property $\Pi$ is
\textit{hereditary} if it is closed under taking induced subgraphs, and it is
\textit{extendable} if every graph $H\in\Pi$ occurs as an induced subgraph of a
member of $\Pi$ on $|V(H)|+m$ vertices for every $m\geq0$.

We work in the \textit{dense graph model}, also called the adjacency-matrix
model, introduced in \cite{goldreich1998property} and surveyed in
\cite{goldreich2017introduction,bhattacharyyaYoshida2022property}.  An input
$G=(V,E)$ is represented by its adjacency predicate
$g:V\times V\to\{0,1\}$, where $g(u,u)=0$, $g(u,v)=g(v,u)$, and
$g(u,v)=1$ precisely when $\{u,v\}\in E$.

\begin{defn}[Distance from a graph property]
If $G$ and $G'$ are graphs on the same $n$-element vertex set, define
\[
\dist(G,G')=\frac{|E(G)\triangle E(G')|}{n^2},
\qquad
\dist(G,\Pi)=\min_{\substack{G'\in\Pi\\V(G')=V(G)}}\dist(G,G').
\]
We say that $G$ is \textit{$\varepsilon$-far from $\Pi$} if
$\dist(G,\Pi)>\varepsilon$; otherwise it is
\textit{$\varepsilon$-close to $\Pi$}.
\end{defn}

The normalization by $n^2$, rather than by $\binom n2$, is immaterial up to an
absolute factor.  We use this convention throughout Part~I.  In the relational
setting of Part~II, each relation is normalized by the number of tuples of its
arity.

A tester has oracle access to the adjacency predicate and to independent
uniform samples from the vertex set.  In the canonical formulation below, a
request for a sample larger than the input is interpreted as a request to
inspect the whole graph.  This convention allows a size-oblivious tester to
decide membership exactly on sufficiently small inputs without being given
their order.

\begin{defn}[Size-oblivious tester]
Let $\Pi$ be a graph property.  A \textit{size-oblivious tester} for $\Pi$ is a
randomized oracle algorithm which, given $\varepsilon>0$ but not $|V(G)|$,
satisfies the following conditions.
\begin{enumerate}[label=\textnormal{(\roman*)}]
\item If $G\in\Pi$, then the tester accepts $G$ with probability at least
$2/3$.
\item If $G$ is $\varepsilon$-far from $\Pi$, then the tester rejects $G$
with probability at least $2/3$.
\end{enumerate}
Its \textit{query complexity} is the maximum number of adjacency queries it
makes, and its \textit{sample complexity} is the maximum number of vertex
labels it inspects.  Both bounds are required to depend only on $\varepsilon$.
The tester has \textit{one-sided error} if it accepts every graph in $\Pi$ with
probability one.
\end{defn}

We tacitly replace complexity bounds by their nonincreasing monotone closures
in the proximity parameter.  A \textit{canonical tester} samples a uniformly
random set $S$ of a prescribed size, queries every pair in $S$, and decides
from the isomorphism type of $G[S]$; if the prescribed sample size exceeds
$|V(G)|$, it inspects the whole graph.  By the canonicalization theorem of
Goldreich and Trevisan \cite{goldreichTrevisan2003three}, a size-oblivious
tester making $Q(\varepsilon)$ adjacency queries can, after changing
$\varepsilon$ by an absolute constant factor, be converted into a canonical
tester sampling $O(Q(\varepsilon))$ vertices.  Conversely, revealing an
induced subgraph on $s$ vertices uses $O(s^2)$ adjacency queries.

For a hereditary property, a canonical one-sided tester may reject exactly
when the sampled induced graph does not belong to the property.  Indeed, every
induced subgraph of a graph in $\Pi$ again lies in $\Pi$, while enlarging the
collection of rejection types outside $\Pi$ can only increase the rejection
probability.

The guiding question of this paper is whether graph properties admitting
size-oblivious one-sided-error testers with prescribed query complexity can be
characterized in combinatorial and quantitative terms.

The qualitative theory is due to Alon and Shapira \cite{AlonShapira}. They proved that every hereditary graph property admits a size-oblivious one-sided-error tester, through an induced removal lemma for possibly infinite families of forbidden induced subgraphs. They also characterized the arbitrary graph properties admitting such testers through the notion of semi-hereditariness. Informally, a property is semi-hereditary if it is contained in a hereditary property $\cH$, and every sufficiently large graph in $\cH$ is already close to the original property.

Thus the qualitative boundary of one-sided testability is well understood. What remains unclear is the quantitative structure controlling the number of queries. Alon, Fischer, Newman, and Shapira asked in Section~8 of \cite{alon2008characterization} for a quantitative version of the Alon--Shapira characterization. More recently, Goldreich \cite{goldreich2021open} highlighted the broader problem of classifying graph properties according to their query complexity, pointing to the size-oblivious model as a particularly natural regime in which such a characterization might be possible.

Our main result gives such a characterization. The answer is expressed through suitable hypergraph containers.

\subsection{From regularity to containers}

The proof of the qualitative theorem of Alon and Shapira is based on a strengthened form of Szemerédi's regularity lemma due to Alon, Fischer, Krivelevich, and Szegedy \cite{alon2000efficient}. This framework is well suited to proving that a tester with complexity independent of the size of the graph exists. It is, however, poorly adapted to the quantitative question considered here.

Regularity-based proofs generally lead to very large quantitative losses. Even when a property is described by finitely many forbidden induced subgraphs, the resulting bounds are typically of tower type. For hypergraph properties, regularity arguments may produce still larger, Ackermann-type dependencies. When a hereditary property is described by an arbitrary, possibly infinite, collection of forbidden induced subgraphs, the compactness step in the proof need not yield an effective bound at all. In particular, regularity does not retain the dependence on an arbitrary prescribed tester complexity.

Analytic and limit approaches give a complementary explanation of the same qualitative phenomenon. The theory of convergent dense graph sequences and testable graph parameters was developed by Borgs, Chayes, Lovász, Sós, Szegedy, and Vesztergombi \cite{BorgsChayesLovaszSosSzegedyVesztergombi}. Lovász and Szegedy \cite{LovaszSzegedy} placed hereditary testing in the setting of sampling-continuous graphon properties, while Elek and Szegedy \cite{ElekSzegedy} developed measure-theoretic and ultraproduct methods for hypergraphs. These approaches explain the generality of hereditary testability, but they are likewise not designed to retain effective dependence on tester complexity.

It is therefore natural to search for a more effective combinatorial mechanism. In extremal combinatorics, hypergraph containers have repeatedly played precisely this role: they replace a global regularity decomposition by a direct description of the entire family of objects avoiding a collection of forbidden configurations.

The hypergraph container method was introduced independently by Balogh, Morris, and Samotij \cite{balogh2015independent} and by Saxton and Thomason \cite{saxton2015hypergraph}, building on earlier work of Kleitman and Winston \cite{kleitman1982number}; see \cite{balogh2018method} for a survey. Over the past decade, it has become one of the main tools of extremal and probabilistic combinatorics.

Informally, a hypergraph container theorem gives a small family of subsets, called containers, such that every independent set of the hypergraph is contained in one of them, while each container spans only few hyperedges. Thus, rather than studying all independent sets individually, one can cover them by a controlled collection of larger sets that are still almost independent.

A standard example is the enumeration of triangle-free graphs. Consider the $3$-uniform hypergraph whose vertex set is $E(K_n)$ and whose hyperedges are the triples of edges forming triangles. Its independent sets are precisely the triangle-free graphs on $[n]$. Applying the container method gives a family of subsets of $E(K_n)$, which we identify with graphs on $[n]$, such that every triangle-free graph is contained in one of these graphs and each container contains few triangles. Moreover, there are only $2^{o(n^2)}$ containers. Supersaturation then implies that every container has at most $n^2/4+o(n^2)$ edges. Since each such container has at most $2^{n^2/4+o(n^2)}$ subgraphs, the number of labelled triangle-free graphs is at most $2^{n^2/4+o(n^2)}$. The matching lower bound follows by taking arbitrary subgraphs of a balanced complete bipartite graph.

The important point for us is that containers turn a quantitative abundance of forbidden configurations into an effective structural description of all objects avoiding them. Our approach follows the same philosophy. If a graph is far from a hereditary property, then it has many small induced subgraphs violating the property. These bad vertex sets form a dense uniform hypergraph, and the induced subgraphs satisfying the property correspond to independent sets in this hypergraph.

Containers have recently also appeared in property testing, where they were used to obtain improved quantitative bounds for particular classes of properties \cite{blais2024new,blais2025testing,seth2025tolerant, grigorescu2026testing}. Their role in the present paper is more structural: rather than merely providing improved testers for particular properties, suitable containers characterize the quantitative structure underlying one-sided testability itself.

\subsection{Ordered containers}

The main idea underlying our characterization is that, inside any graph that is far from a hereditary property, every induced subgraph satisfying the property is forced to lie inside a substantially smaller ambient vertex set, and that this ambient set is determined by a small amount of information. We formalize this through the following notion.

For a set $V$, let $V_{\neq}^{j}$ denote the family of ordered $j$-tuples of distinct elements of $V$.  We write
$(x)_j=x(x-1)\cdots(x-j+1)$ for the falling factorial, with $(x)_0=1$, and use the convention that $\binom ab=0$ when $b<0$ or $b>a$.

\begin{defn}[Ordered container lemma for a hereditary graph property]
\label{def:ordered-container-graph}
Let $\Pi$ be a hereditary graph property. For $\varepsilon,\eta>0$ and integers $N,q\geq0$, we say that $\Pi$ admits an
$(\varepsilon,\eta,N,q)$-\textit{ordered container lemma} if the following holds.

For every $n\geq N$ and every $n$-vertex graph $G=(V,E)$ that is $\varepsilon$-far from $\Pi$, there exist maps
\(f:\{U\subseteq V:G[U]\in\Pi\}\longrightarrow \bigcup_{j=0}^{q}V_{\neq}^{j}, \qquad C:\operatorname{im}(f)\longrightarrow \binom{V}{\leq(1-\eta)n},\)
such that every entry of $f(U)$ belongs to $U$ and $U\subseteq C(f(U))$ whenever $G[U]\in\Pi$. We refer to $f(U)$ as the \textit{fingerprint} of $U$, and to $C(f(U))$ as its \textit{container}.
\end{defn}

The container associated with an induced $\Pi$-subgraph is constructed by an iterative procedure. At each step, the procedure selects a vertex from the induced subgraph and uses it to discard part of the ambient vertex set. The selected vertices form the fingerprint: a short record of the choices made during the construction, from which the final container can be recovered.

The order on the fingerprint records the successive choices made in the construction of the container. It is not inherited from an ordering of the ambient graph. Thus an ordered container lemma asserts that every induced $\Pi$-subgraph of a graph far from $\Pi$ is contained in a vertex set missing at least an $\eta$-fraction of the ambient vertices, and that this containing set is determined by at most $q$ vertices of the induced subgraph itself.

With this notation in place, we may state our main theorem.

\begin{thm}[Characterization of hereditary properties having specified complexity]
\label{thm:hereditary-characterization}
Let $\Pi$ be a hereditary graph property, and let $Q:(0,1)\to\N$. Then the following hold.
\begin{enumerate}[label=\textnormal{(\roman*)}]
\item If $\Pi$ admits a size-oblivious one-sided-error tester with query complexity at most $Q(\varepsilon)$, then there is an absolute constant $c\in(0,1)$ such that, for every $\varepsilon>0$, the property $\Pi$ admits an ordered container lemma with parameters satisfying
\[
q(\varepsilon)=O(Q(c\varepsilon)),\qquad
\eta(\varepsilon)=\Omega\!\left(Q(c\varepsilon)^{-1}\right),
\qquad
N(\varepsilon)=O(Q(c\varepsilon)).
\]

\item Conversely, if $\Pi$ admits ordered container lemmas with parameters $\eta(\varepsilon)$, $N(\varepsilon)$, and $q(\varepsilon)$, then $\Pi$ admits a size-oblivious one-sided-error tester with query complexity
\(\widetilde O\!\left( N(\varepsilon)^2+ \frac{(q(\varepsilon)+1)^2}{\eta(\varepsilon)^2} \right).\)

\end{enumerate}
\end{thm}

Here and throughout, $\widetilde O(\cdot)$ and $\widetilde\Omega(\cdot)$ hide polylogarithmic factors. The theorem shows that one-sided testability and ordered-container parameters are quantitatively equivalent up to explicit polynomial transformations.

The main ingredient in the proof is a removal--container equivalence. Given an $n$-vertex graph $G$ and an integer $s\geq1$, we form the $s$-uniform hypergraph $H_{G,\Pi}$ on $V(G)$ with
\[
E(H_{G,\Pi})
=
\left\{
S\in\binom{V(G)}{s}:G[S]\notin\Pi
\right\}.
\]
Since $\Pi$ is hereditary, every set $U\subseteq V(G)$ satisfying $G[U]\in\Pi$ is an independent set in $H_{G,\Pi}$.

If a canonical tester sampling $s$ vertices rejects every graph that is $\varepsilon$-far from $\Pi$ with probability at least $2/3$, then $e(H_{G,\Pi})\geq \frac23\binom{n}{s}$. The required containers therefore follow from a purely hypergraph-theoretic statement.

\begin{thm}[Dense ordered containers]
\label{thm:dense-ordered-containers-intro}
Let $H$ be an $s$-uniform hypergraph on an $n$-element vertex set $V$, and suppose that $e(H)\geq\delta\binom{n}{s}$. Then there exist maps
\(f:\cI(H)\longrightarrow\bigcup_{j=0}^{s-1}V_{\neq}^{j}, \qquad C:\operatorname{im}(f)\longrightarrow \binom{V}{\leq(1-\delta/s)n},\)
such that every entry of $f(I)$ belongs to $I$ and $I\subseteq C(f(I))$ for every independent set $I\in\cI(H)$.
\end{thm}

The proof is an elementary iterative pivot argument. Starting from an independent set $I$, we choose a vertex $v\in I$ of maximum degree and record it in the fingerprint. We then pass to the link of $v$, namely the hypergraph whose edges are the sets $e\setminus{v}$ over all hyperedges $e$ containing $v$. This reduces the uniformity by one.
The maximality of the degree of $v$ allows us to exclude from the eventual container every vertex whose degree is larger than that of $v$, since none of these vertices can belong to $I$. We repeat the procedure in the successive links. If, at some stage, there are many vertices of degree larger than the chosen pivot, then a positive proportion of the ambient vertices is excluded and the desired shrinkage follows. Otherwise, the chosen pivot has degree comparable to the average degree, so a definite proportion of the edge density survives in the next link. Iterating this alternative, either sufficient shrinkage occurs at an intermediate stage, or enough density remains when the process reaches a $1$-uniform hypergraph. In the latter case, its edges are simply vertices that cannot belong to $I$, and deleting them gives the required container.

The converse direction is a counting argument. Suppose that every independent set is contained in a set of size at most $(1-\eta)n$ determined by a fingerprint of length at most $q$. If $T$ is a uniformly random $t$-subset of the vertex set and $T$ is independent, then $T$ contains its fingerprint and all its remaining vertices lie in the corresponding container. Consequently,
\[
\Pr[T\text{ is independent}]
\leq
\sum_{j=0}^{q}(t)_j
\frac{\binom{\lfloor(1-\eta)n\rfloor-j}{t-j}}
{\binom{n-j}{t-j}}
\leq
\sum_{j=0}^{q}t^j(1-\eta)^{t-j}.
\]
The last expression is at most $1/3$ once
$t=O\!\left(\frac{q+1}{\eta}\log\left(\frac{q+1}{\eta}+2\right)\right)$.
Thus an ordered container lemma yields a removal statement and hence a one-sided tester.

In this way, our proof factors through the equivalence between one-sided testing, induced removal, and ordered containers.

\subsection{Arbitrary graph properties}

The passage from hereditary properties to arbitrary ones is achieved through the notion of semi-hereditariness introduced by Alon and Shapira.

\begin{defn}[Semi-hereditary graph property]
A graph property $\Pi$ is \textit{semi-hereditary} if there exist a hereditary graph property $\cH\supseteq\Pi$ and a function $M:(0,1)\to\N$ such that, for every $\varepsilon>0$, every graph $G\in\cH$ with at least $M(\varepsilon)$ vertices is $\varepsilon$-close to $\Pi$. We call $(\cH,M)$ a \textit{semi-hereditary witness} for $\Pi$.
\end{defn}

Starting from a one-sided tester for $\Pi$, the argument of Alon and Shapira constructs a hereditary envelope $\cH$ by forbidding the finite induced graphs on which the tester can reject. The same tester also tests $\cH$, and every sufficiently large member of $\cH$ must be close to $\Pi$. Conversely, if $\Pi$ is semi-hereditary with witness $(\cH,M)$, then testing $\Pi$ reduces to testing $\cH$, together with exact inspection below the threshold $M$.

Combining this reduction with Theorem~\ref{thm:hereditary-characterization} gives the following.

\begin{cor}[Quantitative extension to arbitrary graph properties]
\label{cor:full-characterization}
Let $\Pi$ be a graph property, and let $Q:(0,1)\to\N$. Then the following hold.
\begin{enumerate}[label=\textnormal{(\roman*)}]
\item If $\Pi$ admits a size-oblivious one-sided-error tester with query complexity at most $Q(\varepsilon)$, then $\Pi$ has a semi-hereditary witness $(\cH,M)$ satisfying $M(\varepsilon)=O(Q(c\varepsilon))$ for an absolute constant $c>0$, and $\cH$ admits ordered-container parameters satisfying
\[
q(\varepsilon)=O(Q(c\varepsilon)),\qquad
\eta(\varepsilon)=\Omega(Q(c\varepsilon)^{-1}),
\qquad
N(\varepsilon)=O(Q(c\varepsilon)).
\]

\item Conversely, if $\Pi$ has a semi-hereditary witness $(\cH,M)$ and $\cH$ admits ordered-container parameters $(\eta,N,q)$, then $\Pi$ admits a size-oblivious one-sided-error tester with query complexity
\(O\!\left(M(\varepsilon/2)^2\right) + \widetilde O\!\left( N(\varepsilon/2)^2+ \frac{(q(\varepsilon/2)+1)^2} {\eta(\varepsilon/2)^2} \right).\)

\end{enumerate}
\end{cor}

Thus the quantitative part of the problem is carried by ordered containers for the hereditary envelope, while semi-hereditariness describes precisely how an arbitrary graph property reduces to that setting.

\subsection{Beyond graphs}

Although graphs are the main setting of the paper and the source of the motivating question, the combinatorial argument above is not graph-specific. Once local obstructions are encoded as the edges of a uniform hypergraph, the proof no longer uses symmetry, irreflexivity, or the fact that the ambient relation has arity two.

The qualitative theory beyond graphs is already broad. Rödl and Schacht \cite{RodlSchacht} extended hereditary testability to uniform hypergraphs by means of hypergraph regularity and removal. Austin and Tao \cite{AustinTao} developed a general framework for multiple directed, coloured, and non-uniform hypergraphs, allowing loops and relations of different arities, and studied the stronger notion of local repairability. Alon, Ben-Eliezer, and Fischer \cite{AlonBenEliezerFischer} obtained related results for edge-coloured ordered graphs and for two-dimensional matrices over finite alphabets.

Our argument gives a common quantitative and regularity-free treatment of these settings. We formulate this in the language of finite relational structures.

\begin{defn}[Finite relational signatures and structures]
\label{def:finite-relational-signature}
 A \textit{finite relational signature} is a finite list $\tau=\{R_1,\dots,R_m\}$ of relation symbols together with a positive integer
$a(R_i)$, called the \textit{arity} of $R_i$, for each $i$.  A finite
$\tau$-structure $\cA$ consists of a finite vertex set $V(\cA)$ and an
interpretation
\(R_i^{\cA}\subseteq V(\cA)^{a(R_i)}\)
of each symbol $R_i$.  We write
$r(\tau)=\max_{R\in\tau}a(R)$ for the maximum arity.
\end{defn}

A signature specifies only the names of the relations and the number of inputs each relation takes. It does not impose any further conditions on those relations. Conditions such as symmetry, the absence of loops, or the requirement that every pair receive exactly one colour are instead included in the definition of the property under consideration.
For example, a simple graph is represented by one binary relation $E$, with the additional requirements that $E$ be symmetric and irreflexive. A directed graph is represented by one binary relation with no symmetry requirement. An edge-coloured graph is represented by one binary relation for each colour, together with the requirement that these relations partition the pairs of distinct vertices. Finally, a $k$-uniform hypergraph is represented by one $k$-ary relation, required to be invariant under permutations of its arguments and to contain only tuples of distinct vertices.
In this way, the same formalism treats directed, coloured, uniform, and non-uniform structures within a single framework.

For $U\subseteq V(\cA)$, the \textit{induced substructure} $\cA[U]$ is obtained
by restricting every relation to tuples all of whose entries belong to $U$.

\begin{defn}[Relational properties]
A $\tau$-property is a family of finite $\tau$-structures closed under isomorphism. It is \textit{hereditary} if $\cA\in\cP$ implies $\cA[U]\in\cP$ for every $U\subseteq V(\cA)$. It is \textit{extendable} if, for every $\cA\in\cP$ and every $m\geq0$, there is a structure $\cB\in\cP$ on $|V(\cA)|+m$ vertices containing $\cA$ as an induced substructure.
\end{defn}

If $\cA$ and $\cB$ are $\tau$-structures on the same $n$-element vertex set, we define
\(d_\tau(\cA,\cB) = \frac{1}{|\tau|} \sum_{R\in\tau} \frac{|R^{\cA}\triangle R^{\cB}|}{n^{a(R)}}.\)
Distance from a property and the notions of being $\varepsilon$-far or $\varepsilon$-close are defined in the usual way.

The intrinsic testing notion is again induced sampling.

\begin{defn}[Oblivious induced-sampling tester]
Let $\cP$ be a $\tau$-property. An \textit{oblivious induced-sampling tester} for $\cP$ consists, for every $\varepsilon>0$, of an integer $s(\varepsilon)$ and a collection $\cR_\varepsilon$ of isomorphism types of $\tau$-structures on $s(\varepsilon)$ vertices.

On an input $\cA$ with $|V(\cA)|\geq s(\varepsilon)$, the tester chooses a uniformly random set $S\in\binom{V(\cA)}{s(\varepsilon)}$, reveals the complete induced structure $\cA[S]$, and rejects if and only if $\cA[S]\in\cR_\varepsilon$. If $|V(\cA)|<s(\varepsilon)$, it may inspect the whole input.

The tester has one-sided error if it always accepts members of $\cP$, and is sound if it rejects every $\varepsilon$-far input with probability at least $2/3$. Its sample complexity is $s(\varepsilon)$.
\end{defn}

For hereditary properties, the tester may reject exactly when the sampled induced structure does not belong to the property.

\begin{defn}[Ordered containers for relational properties]
Let $\cP$ be a hereditary $\tau$-property. For $\varepsilon,\eta>0$ and integers $N,q\geq0$, we say that $\cP$ admits an
$(\varepsilon,\eta,N,q)$-\textit{ordered container lemma} if, for every $n\geq N$ and every $n$-vertex $\tau$-structure $\cA$ that is $\varepsilon$-far from $\cP$, there exist maps
\[
f:\{U\subseteq V(\cA):\cA[U]\in\cP\}
\longrightarrow
\bigcup_{j=0}^{q}V(\cA)_{\neq}^{j},
\qquad
C:\operatorname{im}(f)
\longrightarrow
\binom{V(\cA)}{\leq(1-\eta)n},
\]
such that every entry of $f(U)$ belongs to $U$ and $U\subseteq C(f(U))$.
\end{defn}

\begin{thm}[Container characterization for hereditary relational structures]
\label{thm:relational-hereditary-characterization-intro}
Let $\cP$ be a hereditary property of finite $\tau$-structures.
\begin{enumerate}[label=\textnormal{(\roman*)}]
\item If $\cP$ has a one-sided induced-sampling tester with sample complexity $s(\varepsilon)$, then it admits ordered-container parameters
\(N(\varepsilon)=s(\varepsilon)\), \(q(\varepsilon)=s(\varepsilon)-1\), and \(\eta(\varepsilon)={2}/{(3s(\varepsilon))}\).

\item Conversely, if $\cP$ admits ordered-container parameters $(\eta,N,q)$, then it has a one-sided induced-sampling tester with sample complexity
\[
O\!\left(
\max\left\{
N(\varepsilon),
\frac{q(\varepsilon)+1}{\eta(\varepsilon)}
\log\left(
\frac{q(\varepsilon)+1}{\eta(\varepsilon)}+2
\right)
\right\}
\right).
\]

\end{enumerate}
\end{thm}

Revealing the complete induced structure on $t$ vertices requires at most $\sum_{R\in\tau}t^{a(R)}\leq|\tau|t^r$ relation queries. Thus the theorem also gives a characterization in terms of relation-query complexity.

\begin{defn}[Semi-hereditary relational property]
A $\tau$-property $\cP$ is \textit{semi-hereditary} if there exist a hereditary property $\cH\supseteq\cP$ and a function $M:(0,1)\to\N$ such that every $\cA\in\cH$ with $|V(\cA)|\geq M(\varepsilon)$ is $\varepsilon$-close to $\cP$. We call $(\cH,M)$ a \textit{semi-hereditary witness}.
\end{defn}

The same hereditary-envelope argument gives a quantitative characterization
of arbitrary relational properties.  For hereditary properties in a fixed
bounded-arity signature, polynomial induced-sampling complexity, polynomial
relation-query complexity, and polynomial ordered-container parameters are
equivalent.  For an arbitrary property, the corresponding statement also
requires polynomial control of the semi-hereditary threshold.

\subsection{Applications}

The characterization has three consequences that we develop in Part~III.

\subsection*{Partition properties.}
Let $\cP_1,\dots,\cP_k$ be properties of structures in a common binary
signature.  We write $\Part(\cP_1,\dots,\cP_k)$ for the property of admitting a
partition $V(\cA)=V_1\cup\cdots\cup V_k$ such that
$\cA[V_i]\in\cP_i$ for every $i$; tuples meeting distinct parts are
unrestricted.  We prove that if the $\cP_i$ are nonempty, hereditary,
extendable, and one-sided testable, then their partition property is also
one-sided testable, with an explicit quantitative bound.

\begin{thm}[Partition closure for binary signatures]
\label{thm:relational-partition-closure-intro}
Let every relation symbol in $\tau$ have arity two, and let
$\cP_1,\dots,\cP_k$ be nonempty hereditary and extendable $\tau$-properties.
Suppose that $\cP_i$ has one-sided induced-sampling complexity $S_i(\delta)$,
and put $S(\delta)=\max_{i\in[k]}S_i(\delta)$.  Then
$\Part(\cP_1,\dots,\cP_k)$ has a one-sided induced-sampling tester with sample
complexity
\(\widetilde O\!\left( \frac{k}{\varepsilon} S\!\left(\frac{\varepsilon}{8k}\right)^2 \right).\)
\end{thm}

For graphs, this applies to $(r,s)$-colourability, where the vertex set is
partitioned into $r$ independent sets and $s$ cliques.  It includes ordinary
colourability and split graphs, and it gives a tester for bounded cochromatic
number, the minimum number of parts each inducing either a clique or an
independent set.  For digraphs it gives a tester for bounded dichromatic
number, the minimum number of acyclic parts.

\subsection*{Linearly large induced substructures.}
For $\rho\in(0,1]$ and a property $\cP$, let
\[
\mathrm{Large}_\rho(\cP)
=
\left\{
\cA:
\begin{array}{c}
\text{there exists }U\subseteq V(\cA)\text{ with }|U|\geq\rho|V(\cA)|\\[-2pt]
\text{and }\cA[U]\in\cP
\end{array}
\right\}.
\]
This property is generally not hereditary and need not be a bounded partition
property.  Nevertheless, iterating the ordered-container construction gives a
tester whenever $\cP$ is hereditary, extendable, and one-sided testable.

\begin{thm}[Testing for a large induced member]
\label{thm:relational-large-induced-intro}
Let $\tau$ be a finite binary signature, let $\cP$ be a nonempty hereditary and
extendable $\tau$-property, and let $S_{\cP}(\delta)$ denote its one-sided
induced-sampling complexity.  Fix $\rho\in(0,1]$ and
$0<\varepsilon\leq\rho^2/8$, and put $\delta=\varepsilon/(4\rho^2)$.  Then
$\mathrm{Large}_\rho(\cP)$ has a two-sided induced-sampling tester using
\(\widetilde O\!\left( \frac{\rho^2}{\varepsilon^2} S_{\cP}(\delta)^2 \log\frac2\rho \right)\)
vertices.
\end{thm}

In particular, we obtain a tester for the existence of an induced acyclic
subdigraph containing a prescribed positive proportion of the vertices.

\subsection*{Counting inside far hosts.}
Ordered containers also control all induced members of a hereditary property
inside a fixed far host.  If every induced $\Pi$-subgraph of an $n$-vertex
graph $G$ has a fingerprint of length at most $q$ and a container of size at
most $(1-\eta)n$, then
\(\bigl|\{U\subseteq V(G):G[U]\in\Pi\}\bigr| \leq (q+1)n^q2^{(1-\eta)n}.\)
Thus the family of vertex sets inducing members of $\Pi$ has a uniform entropy
deficit.  Combining this observation with known edit-distance results for
hereditary properties yields structural statements in random hosts; for
example, with high probability every perfect induced subgraph of $G(n,1/2)$ is
contained in one of polynomially many vertex sets of size at most
$(1-\eta)n$, for a constant $\eta>0$.  The edit-distance background and the
formal random-host consequences are deferred to Section~\ref{sec:far-host-applications}.

\subsection{Organization of the paper}

The paper is divided into three parts.  Part~I proves the graph
characterization: first for hereditary properties through the
removal--container equivalence, and then for arbitrary properties through
semi-hereditary envelopes.  Part~II develops the corresponding theory for
finite relational structures of bounded arity, including the induced-sampling
model and its canonicalization.  Part~III contains the applications: closure
under vertex partitions, properties defined by a linearly large induced
substructure, and counting and structural consequences inside far hosts.

\subsection*{Acknowledgements}
This project was carried out while the first two authors were participating in the 2025 NII International Internship Program.  The authors are grateful to the program for its support.  The first author is supported by the CB European PhD Studentship funded by Trinity College, Cambridge.

\part{Graph properties}

All graph-theoretic and testing terminology in this part is as fixed in the
introduction.

\section{The graph characterization}
\label{sec:graph-characterization}
We suppress floor and ceiling signs when they have no effect on the argument.
\subsection{Hereditary properties}

We begin by isolating the combinatorial heart of the argument, namely the equivalence between removal statements and ordered container statements. The forward implication constructs, from a graph that is far from the property, a shrinking container around every induced $\Pi$-subgraph by iteratively exposing high-degree vertices in an auxiliary hypergraph. The reverse implication is a counting argument: once every induced $\Pi$-subgraph is controlled by a short fingerprint and a substantially smaller container, a random induced subgraph has a positive probability of witnessing a violation of~$\Pi$.

\begin{defn}[Removal lemma for a hereditary graph property]
Let $\Pi$ be a hereditary graph property. For real numbers $\varepsilon,\delta>0$ and integers $N,s\ge 1$, we say that $\Pi$ admits an $(\varepsilon,\delta,N,s)$-removal lemma if the following holds. For every integer $n\ge N$ and every $n$-vertex graph $G=(V,E)$ that is $\varepsilon$-far from $\Pi$, there are at least $\delta\binom{n}{s}$ sets $S\in\binom{V}{s}$ such that $G[S]\notin \Pi$.
\end{defn}

\begin{thm}[Removal--container equivalence]\label{thm:removal-container-equivalence}
Let $\Pi$ be a hereditary graph property.
\begin{enumerate}
\item If $\Pi$ admits an $(\varepsilon,\delta,N,s)$-removal lemma, then it
admits the following ordered container statement: for every integer
$n\geq\max\{N,s\}$ and every $n$-vertex graph $G=(V,E)$ that is
$\varepsilon$-far from $\Pi$, there are maps
\(f:\{U\subseteq V:G[U]\in\Pi\} \longrightarrow \bigcup_{j=0}^{s-1}V_{\neq}^{j}, \qquad C:\operatorname{im}(f) \longrightarrow \binom{V}{\leq(1-\delta/s)n},\)
such that every entry of $f(U)$ belongs to $U$ and
$U\subseteq C(f(U))$ whenever $G[U]\in\Pi$.

\item Conversely, suppose that $\Pi$ admits an ordered container statement
with proximity parameter $\varepsilon$, shrinkage $\eta$, threshold $N$,
and fingerprint bound $q$.  Then, for every integer $t\geq q$ such that
\(\sum_{j=0}^{q}t^j(1-\eta)^{t-j}<1,\)
the property $\Pi$ admits an $(\varepsilon,\delta_t,N,t)$-removal lemma,
where
\(\delta_t=1-\sum_{j=0}^{q}t^j(1-\eta)^{t-j}>0.\)
\end{enumerate}
\end{thm}

\begin{proof}
We first prove the forward implication.  Let $G=(V,E)$ be an $n$-vertex graph
with $n\geq\max\{N,s\}$ that is $\varepsilon$-far from $\Pi$, and let $H$ be the
$s$-uniform hypergraph on $V$ whose edges are the sets $S\in\binom Vs$ for
which $G[S]\notin\Pi$.  By hypothesis,
$e(H)\geq\delta\binom ns$.  Every set $U\subseteq V$ with $G[U]\in\Pi$ is an
independent set of $H$.

Fix a total order on $V$, used only to make all choices canonical.  If $s=1$,
let every independent set have the empty fingerprint and put
$C(())=V\setminus E(H)$, where the edge set of the $1$-uniform hypergraph is
identified with a subset of $V$.  Then
$|C(())|\leq(1-\delta)n$, and the conclusion follows.  We may therefore assume
$s\geq2$.

If $n=s$, then $H$ is nonempty and its unique possible edge is $V$.  Hence
every independent set is a proper subset of $V$.  Assign to an independent set
its increasing ordering as fingerprint and use the set itself as its
container.  Since $\delta\leq1$, every such container has size at most
$n-1\leq(1-\delta/s)n$.

Assume henceforth that $n>s$.  If an independent set $U$ has
$|U|\leq s-2$, let $f(U)$ be its increasing ordering and set $C(f(U))=U$.
These fingerprints have length at most $s-2$.  Now let $|U|\geq s-1$.
Starting from $H^{(s)}=H$, define hypergraphs
$H^{(s)},H^{(s-1)},\ldots,H^{(1)}$ recursively.  For
$i=s,s-1,\ldots,2$, choose a vertex
$v_i\in U\cap V(H^{(i)})$ of maximum degree among the vertices of
$U\cap V(H^{(i)})$, breaking ties by the fixed order, and put
\(X_i=\{u\in V(H^{(i)}): \deg_{H^{(i)}}(u)>\deg_{H^{(i)}}(v_i)\}.\)
Let $H^{(i-1)}$ be the $(i-1)$-uniform link of $v_i$ in $H^{(i)}$, on the
vertex set $V(H^{(i)})\setminus\{v_i\}$.  Since $U$ is independent in
$H^{(i)}$, the set $U\cap V(H^{(i-1)})$ is independent in $H^{(i-1)}$; in
particular, the construction is well defined through the last step.

Set
\(f(U)=(v_s,v_{s-1},\ldots,v_2)
\)
and define
\(C(f(U)) =V\setminus\left(E(H^{(1)})\cup\bigcup_{i=2}^{s}X_i\right).\)
Here again $E(H^{(1)})$ is viewed as a set of vertices.  This is a genuine
function of the fingerprint: starting from $H$, the ordered tuple
$(v_s,\ldots,v_2)$ uniquely reconstructs all links $H^{(i)}$, all sets $X_i$,
and hence the displayed container.  Thus two independent sets with the same
fingerprint receive the same container.  Moreover, fingerprints arising in
this case have length exactly $s-1$, so they cannot coincide with the
fingerprints used for sets of size at most $s-2$.

By maximality of $v_i$ inside $U\cap V(H^{(i)})$, every $X_i$ is disjoint from
$U$.  Also, $U\cap V(H^{(1)})$ avoids $E(H^{(1)})$.  Consequently
$U\subseteq C(f(U))$.

It remains to prove the shrinkage.  Write
$n_i=|V(H^{(i)})|=n-s+i$ and $m_i=e(H^{(i)})$.  Then
$m_{i-1}=\deg_{H^{(i)}}(v_i)$ for $2\leq i\leq s$.  Let $k$ be the least
index in $[s]$ such that either $k=1$ or
\(m_{k-1}<\frac{k-1}{n_k-1}m_k.\)
For every $j=k+1,\ldots,s$ we therefore have
$m_{j-1}\geq\frac{j-1}{n_j-1}m_j$, and hence
\[
m_k
\geq
\left(\prod_{j=k+1}^{s}\frac{j-1}{n_j-1}\right)m_s
\geq
\frac{\delta n}{s}\binom{n_k-1}{k-1}.
\]
If $k=1$, then $|E(H^{(1)})|=m_1\geq\delta n/s$, which gives the required
bound on $|C(f(U))|$.

Suppose that $k\geq2$ and put
\[
Y=\left\{u\in V(H^{(k)}):
\deg_{H^{(k)}}(u)>\frac{k-1}{n_k-1}m_k\right\}.
\]
Since $n_k\geq k+1$, the average degree $km_k/n_k$ is strictly larger than
the threshold in this definition.  Furthermore,
\(|Y|\geq\frac{m_k}{\binom{n_k-1}{k-1}}.\)
Indeed, otherwise the contribution to the degree sum from $Y$ is less than
$m_k$, while the remaining at most $n_k-1$ vertices contribute at most
$(k-1)m_k$, contradicting that the total degree sum is $km_k$.  By the choice
of $k$,
$\deg_{H^{(k)}}(v_k)=m_{k-1}<\frac{k-1}{n_k-1}m_k$, so
$Y\subseteq X_k$.  It follows that
\[
|X_k|\geq |Y|
\geq\frac{m_k}{\binom{n_k-1}{k-1}}
\geq\frac{\delta n}{s}.
\]
Thus in every case $|C(f(U))|\leq(1-\delta/s)n$, proving the first
implication.

For the converse, let $G=(V,E)$ be an $n$-vertex graph with $n\geq N$ that is
$\varepsilon$-far from $\Pi$.  The assertion is vacuous when $n<t$, so assume
$n\geq t$ and choose $T\in\binom Vt$ uniformly.  If $G[T]\in\Pi$, then for
some $j\leq q$ its fingerprint $F=f(T)$ is an ordered $j$-tuple contained in
$T$, and
\(T\subseteq C(F)\), where $|C(F)|\leq(1-\eta)n$.  For a fixed ordered tuple
$F$ of length $j$,
\(\Pr[F\subseteq T\subseteq C(F)] =\frac{\binom{|C(F)|-j}{t-j}}{\binom nt}.\)
There are at most $(n)_j$ possible ordered tuples of length $j$.  Summing over
all fingerprints and using
\(\frac{(n)_j}{\binom nt} =\frac{(t)_j}{\binom{n-j}{t-j}}\)
gives
\begin{align*}
\Pr[G[T]\in\Pi]
&\leq
\sum_{j=0}^{q}(t)_j
\frac{\binom{\lfloor(1-\eta)n\rfloor-j}{t-j}}
{\binom{n-j}{t-j}}\\
&\leq
\sum_{j=0}^{q}t^j(1-\eta)^{t-j}.
\end{align*}
Therefore at least a $\delta_t$-fraction of the $t$-subsets induce graphs
outside $\Pi$, as required.
\end{proof}

We now deduce the main characterization theorem from Theorem~\ref{thm:removal-container-equivalence} together with the standard canonical reduction for one-sided testers in the dense graph model.

\begin{proof}[Proof of Theorem \ref{thm:hereditary-characterization}]
Suppose first that $\Pi$ has a size-oblivious one-sided tester of query
complexity at most $Q(\varepsilon)$.  By the canonicalization theorem, there
are absolute constants $C>0$ and $c\in(0,1)$ such that, at proximity
$\varepsilon$, a canonical one-sided tester samples
$s\leq C Q(c\varepsilon)$ vertices.  Enlarge $s$ to an integer if necessary.
If an $n$-vertex graph $G$ is $\varepsilon$-far from $\Pi$ and $n\geq s$,
then at least a $2/3$-fraction of its induced $s$-vertex subgraphs lie outside
$\Pi$.  Thus $\Pi$ admits an $(\varepsilon,2/3,s,s)$-removal lemma.  The first
part of Theorem~\ref{thm:removal-container-equivalence} gives ordered
containers with
\(q(\varepsilon)=s-1\),
\(\eta(\varepsilon)=\frac{2}{3s}\),
and \(N(\varepsilon)=s\), which proves the stated bounds.

Conversely, fix $\varepsilon>0$ and abbreviate
$q=q(\varepsilon)$, $\eta=\eta(\varepsilon)$, and $N=N(\varepsilon)$.  By
Lemma~\ref{lem:counting-from-containers} below, there is an absolute constant
$C$ such that
\(t=\left\lceil \frac{C(q+1)}{\eta} \log\left(\frac{q+1}{\eta}+2\right) \right\rceil\)
satisfies
\(\sum_{j=0}^{q}t^j(1-\eta)^{t-j}\leq1/3\).
Set $T=\max\{N,t\}$.  The tester requests a uniformly random set of $T$
vertices; under our standing convention, if the input has fewer than $T$
vertices, it receives and inspects the whole input.  It rejects exactly when
the revealed induced graph does not belong to $\Pi$.

The tester is one-sided by heredity.  On an input with fewer than $T$ vertices
it decides membership exactly.  If $n\geq T$ and the input is
$\varepsilon$-far from $\Pi$, the second part of
Theorem~\ref{thm:removal-container-equivalence} shows that a uniformly random
$t$-set is bad with probability at least $2/3$.  A uniformly random $T$-set
can be generated first and then a uniformly random $t$-subset chosen inside
it.  Since $\Pi$ is hereditary, a good $T$-set contains only good $t$-sets;
hence the probability that the sampled $T$-set is good is at most $1/3$.
The query complexity is
\(O(T^2) =\widetilde O\!\left( N(\varepsilon)^2+ \frac{(q(\varepsilon)+1)^2}{\eta(\varepsilon)^2} \right).\)
\end{proof}

\subsection{Semi-hereditary envelopes}

We now pass from hereditary properties to arbitrary graph properties via the Alon--Shapira notion of semi-hereditariness. 

\begin{defn}[Semi-hereditary graph property]
A graph property $\Pi$ is said to be \textit{semi-hereditary} if there exist a hereditary graph property $H\supseteq \Pi$ and a function $M:(0,1)\to\N$ such that, for every $\varepsilon>0$, every graph $G\in H$ with at least $M(\varepsilon)$ vertices is $\varepsilon$-close to $\Pi$.
\end{defn}

The forward implication combines their hereditary-envelope construction with Theorem~\ref{thm:hereditary-characterization}. For the converse, we test the hereditary envelope first and then use a brute-force routine below the semi-hereditary threshold.

\begin{proof}[Proof of Corollary \ref{cor:full-characterization}]
Suppose first that $\Pi$ has a size-oblivious one-sided tester of query
complexity at most $Q(\varepsilon)$.  By canonicalization, there are absolute
constants $C>0$ and $c\in(0,1)$ and, for each $\varepsilon>0$, a canonical
one-sided tester $T_\varepsilon$ sampling at most
$s(\varepsilon)=C Q(c\varepsilon)$ vertices.  Let $\cB$ be the family of all
finite graphs that occur as a rejection type of some $T_\varepsilon$, and let
$H$ be the hereditary property of containing no member of $\cB$ as an induced
subgraph.

One-sidedness gives $\Pi\subseteq H$.  For each fixed $\varepsilon$, the
same tester $T_\varepsilon$ tests $H$: it never rejects an input in $H$, while
an input that is $\varepsilon$-far from $H$ is also $\varepsilon$-far from
$\Pi$.  Moreover, if $G\in H$, $|V(G)|\geq s(\varepsilon)$, and $G$ were
$\varepsilon$-far from $\Pi$, then $T_\varepsilon$ would reject with positive
probability and hence expose a member of $\cB$ inside $G$, a contradiction.
Thus $(H,M)$ is a semi-hereditary witness with
$M(\varepsilon)\leq s(\varepsilon)$.  Applying
Theorem~\ref{thm:hereditary-characterization} to $H$ gives the asserted
ordered-container parameters.

Conversely, suppose that $(H,M)$ is a semi-hereditary witness for $\Pi$ and
that $H$ has ordered-container parameter functions $(\eta,N,q)$.  Apply
Theorem~\ref{thm:hereditary-characterization} to obtain a canonical one-sided
tester for $H$ at proximity $\varepsilon/2$, with sample size
\(S_H=\widetilde O\!\left( N(\varepsilon/2)+ \frac{q(\varepsilon/2)+1}{\eta(\varepsilon/2)} \right).\)
Set
\(T=\max\{M(\varepsilon/2),S_H\}\).  Request a uniformly random set of $T$
vertices.  If the input has fewer than $T$ vertices, inspect it completely and
decide membership in $\Pi$ exactly.  Otherwise run the tester for $H$ on the
appropriate subset of the sampled vertices.

Inputs in $\Pi$ are accepted because $\Pi\subseteq H$.  Now let $G$ be
$\varepsilon$-far from $\Pi$ and have at least $T$ vertices.  Then $G$ is
$\varepsilon/2$-far from $H$: otherwise there would be a graph $G'\in H$ on
the same vertex set with $\dist(G,G')\leq\varepsilon/2$, while
semi-hereditariness would make $G'$ $\varepsilon/2$-close to $\Pi$, contrary
to the triangle inequality.  Hence the tester for $H$ rejects with probability
at least $2/3$.  The resulting query complexity is
\(O\!\left(M(\varepsilon/2)^2\right) +\widetilde O\!\left( N(\varepsilon/2)^2+ \frac{(q(\varepsilon/2)+1)^2}{\eta(\varepsilon/2)^2} \right).\)
\end{proof}

\part{Bounded-arity relational structures}

The graph argument rests only on an ordered-container lemma for independent
sets in a dense uniform hypergraph.  It therefore does not depend on the
ambient object having a single symmetric irreflexive binary relation.  We now
work over an arbitrary finite relational signature in the sense of
Definition~\ref{def:finite-relational-signature}.

\section{Finite relational structures and induced sampling}
\label{sec:relational-preliminaries}

Fix a finite relational signature $\tau$ and write
$r=r(\tau)=\max_{R\in\tau}a(R)$.  A finite $\tau$-structure $\cA$ has vertex set
$V(\cA)$ and relations $R^{\cA}\subseteq V(\cA)^{a(R)}$ for $R\in\tau$.

\begin{defn}[Relational properties]
For $U\subseteq V(\cA)$, the \textit{induced substructure} $\cA[U]$ is obtained
by restricting every relation to tuples all of whose entries lie in $U$.  A
$\tau$-property is a family of finite $\tau$-structures closed under
isomorphism.  It is \textit{hereditary} if $\cA\in\cP$ implies
$\cA[U]\in\cP$ for every $U\subseteq V(\cA)$, and it is \textit{extendable} if,
for every $\cA\in\cP$ and every $m\geq0$, there is a structure $\cB\in\cP$ on
$|V(\cA)|+m$ vertices containing $\cA$ as an induced substructure.
\end{defn}

\begin{defn}[Distance between relational structures]
If $\cA$ and $\cB$ are $\tau$-structures on the same $n$-element vertex set,
define
\begin{equation}
\label{eq:relational-distance}
d_\tau(\cA,\cB)
=\frac1{|\tau|}\sum_{R\in\tau}
\frac{|R^{\cA}\triangle R^{\cB}|}{n^{a(R)}}.
\end{equation}
We write $d_\tau(\cA,\cP)$ for the minimum of $d_\tau(\cA,\cB)$ over all
$\cB\in\cP$ on the same vertex set, and say that $\cA$ is
\textit{$\varepsilon$-far from $\cP$} if this minimum is greater than
$\varepsilon$; otherwise it is \textit{$\varepsilon$-close to $\cP$}.  For a
simple graph encoded as a symmetric irreflexive binary relation, this differs
from the normalization in Part~I only by an absolute constant.
\end{defn}

The intrinsic testing notion in this setting is induced sampling.

\begin{defn}[Oblivious induced-sampling tester]
Let $\cP$ be a $\tau$-property.  An \textit{oblivious induced-sampling tester}
for $\cP$ consists, for every $\varepsilon>0$, of an integer
$s(\varepsilon)$ and a collection $\cR_\varepsilon$ of isomorphism types of
$\tau$-structures on $s(\varepsilon)$ vertices.  On an input $\cA$ with
$n\geq s(\varepsilon)$, the tester chooses a uniformly random
$s(\varepsilon)$-subset $S\subseteq V(\cA)$, reveals the complete induced
structure $\cA[S]$, and rejects if and only if
$\cA[S]\in\cR_\varepsilon$.  When $n<s(\varepsilon)$, it may inspect the
whole input and decide membership exactly.

The tester has one-sided error if it always accepts structures in $\cP$, and
is sound if it rejects every $\varepsilon$-far input with probability at least
$2/3$.  Its sample complexity is $s(\varepsilon)$.
\end{defn}

For hereditary properties, the rejection rule may always be taken to be the
complement of the property on the sampled set.

\begin{lem}\label{lem:relational-membership-rule}
Let $\cP$ be hereditary and suppose that it has a one-sided induced-sampling
tester with sample complexity $s(\varepsilon)$.  Then the tester which samples
the same number of vertices and rejects exactly when
$\cA[S]\notin\cP$ is also one-sided and has at least the same rejection
probability on every input.
\end{lem}

\begin{proof}
Every isomorphism type on which the original tester can reject lies outside
$\cP$.  Indeed, if a structure $\cF\in\cP$ were a rejection type, then running
the tester on $\cF$ itself would contradict one-sidedness.  Hence replacing
the original rejection family by the whole complement of $\cP$ can only
increase the rejection probability.  Heredity ensures that no induced
substructure of an input in $\cP$ lies outside $\cP$.
\end{proof}

We shall also use the dense relation-query model corresponding to Part~I.
The tester is not given the order of the input; it obtains vertex labels from
an independent uniform-sampling oracle, and a relation query specifies a
relation symbol together with a tuple of previously sampled labels.  Its
complexity may depend on the proximity parameter and on the fixed signature,
but not on the input order.  The canonicalization theorem of Goldreich and
Trevisan \cite{goldreichTrevisan2003three} extends without change from graphs
to a fixed finite relational signature; we record the precise consequence
used below.

\begin{thm}[Canonicalization for bounded-arity structures]
\label{thm:relational-canonicalization}
Let $\tau$ be a fixed finite signature.  There are constants
$c_\tau,C_\tau>0$ such that the following holds.  If a $\tau$-property has a
size-oblivious one-sided tester making at most $Q(\varepsilon)$ relation
queries, then it has a one-sided induced-sampling tester sampling at most
\(C_\tau Q(c_\tau\varepsilon)\)
vertices.  If the property is hereditary, the induced sampler may reject
exactly when the sampled induced structure does not satisfy the property.
\end{thm}

\begin{proof}
The random-relabeling proof of Goldreich and Trevisan uses only permutation
invariance of the input and the fact that a query transcript mentions a
bounded number of vertex labels.  Both features remain valid for a fixed
relational signature.  Indeed, if $r$ is the maximum arity, a transcript of
$Q$ relation queries involves at most $rQ$ distinct labels.  Conditional on
its equality pattern, a uniformly random relabelling sends these labels to a
uniformly random ordered tuple of distinct vertices.  Hence revealing the
complete induced structure on a uniformly random set of $O_\tau(Q)$ vertices
contains all relation values needed to reproduce the transcript, including
adaptive queries.  The treatment of collisions and inputs whose order is
comparable with $Q$, as well as the amplification and constant rescaling of
the proximity parameter, is exactly the same as in the graph proof.  This
gives constants $c_\tau,C_\tau$ depending only on the signature.  Random
relabeling preserves distance from an isomorphism-invariant property and
preserves one-sidedness.  The final assertion follows from
Lemma~\ref{lem:relational-membership-rule}.
\end{proof}

Revealing the complete induced structure on $t$ vertices uses at most
\begin{equation}
\label{eq:relation-query-cost}
\sum_{R\in\tau}t^{a(R)}\leq |\tau|t^r
\end{equation}
relation queries.  Thus sample complexity is the natural arity-independent
parameter, while relation-query complexity is recovered by taking the
$r$-th power, up to a constant depending on the signature.

\section{The relational characterization}
\label{sec:relational-characterization}

We first isolate the pure hypergraph statement already proved inside
Theorem~\ref{thm:removal-container-equivalence}.  If $H$ is a hypergraph, let
$\cI(H)$ denote its family of independent sets.

\begin{lem}[Dense ordered containers]
\label{lem:dense-ordered-containers}
Let $H$ be an $s$-uniform hypergraph on an $n$-element vertex set $V$, where
$1\leq s\leq n$, and suppose that
\(e(H)\geq\delta\binom{n}{s}\).
Then there are a family $\cF$ of ordered tuples of distinct vertices, each of
length at most $s-1$, and maps
\(f:\cI(H)\to\cF,  C:\cF\to\binom{V}{\leq(1-\delta/s)n},\)
such that every entry of $f(I)$ belongs to $I$ and
\(I\subseteq C(f(I))\)
for every $I\in\cI(H)$.
\end{lem}

\begin{proof}
This is precisely the ordered pivot argument in the first implication of
Theorem~\ref{thm:removal-container-equivalence}.  That proof begins with an
arbitrary dense $s$-uniform hypergraph, chooses successive maximum-degree
pivots from a fixed independent set, passes to links, and obtains a container
omitting at least $\delta n/s$ vertices.  No property of the host graph is used
at any point; the graph enters there only to define the obstruction
hypergraph.  Therefore the same construction, with $H$ as the initial
hypergraph, gives the asserted maps.
\end{proof}

The reverse implication is also purely hypergraph-theoretic.

\begin{lem}[Counting from ordered containers]
\label{lem:counting-from-containers}
Suppose that an $n$-vertex hypergraph $H$ admits maps
\(f:\cI(H)\to\bigcup_{j=0}^{q}V_{\neq}^{j}, C:\operatorname{im}(f)\to\binom{V}{\leq(1-\eta)n},\)
such that every entry of $f(I)$ belongs to $I$ and
$I\subseteq C(f(I))$.  Then, for every $q\leq t\leq n$,
\begin{align}
\frac{|\cI(H)\cap\binom{V}{t}|}{\binom{n}{t}}
&\leq
\sum_{j=0}^{q}(t)_j
\frac{\binom{\lfloor(1-\eta)n\rfloor-j}{t-j}}
{\binom{n-j}{t-j}}                                      \label{eq:relational-exact-counting}\\
&\leq\sum_{j=0}^{q}t^j(1-\eta)^{t-j}.                       \label{eq:relational-rough-counting}
\end{align}
Consequently, there is an absolute constant $C>0$ such that the right-hand
side of \eqref{eq:relational-rough-counting} is at most $1/3$ whenever
\begin{equation}
\label{eq:relational-sampling-threshold}
t\geq
\frac{C(q+1)}{\eta}
\log\left(\frac{q+1}{\eta}+2\right).
\end{equation}
\end{lem}

\begin{proof}
Choose $T\in\binom{V}{t}$ uniformly.  If $T$ is independent and its fingerprint
has length $j$, then $f(T)\subseteq T\subseteq C(f(T))$.  For a fixed ordered
tuple $F$ of length $j$,
\(\Pr[F\subseteq T\subseteq C(F)] =\frac{\binom{|C(F)|-j}{t-j}}{\binom{n}{t}}.\)
There are at most $(n)_j$ possible ordered tuples of length $j$.  Summing over
$j$ gives \eqref{eq:relational-exact-counting}, using
\(\frac{(n)_j}{\binom{n}{t}} =\frac{(t)_j}{\binom{n-j}{t-j}}.\)
Moreover,
\(\frac{\binom{\lfloor(1-\eta)n\rfloor-j}{t-j}} {\binom{n-j}{t-j}} \leq(1-\eta)^{t-j},  (t)_j\leq t^j,\)
which proves \eqref{eq:relational-rough-counting}.  Finally, using
$1-\eta\leq e^{-\eta}$, each summand is at most
$\exp(j\log t-\eta(t-j))$.  If $t$ satisfies
\eqref{eq:relational-sampling-threshold} with a sufficiently large absolute
constant, then the sum of the $q+1$ terms is at most $1/3$.
\end{proof}

We now formulate the relational version of the ordered-container condition.

\begin{defn}[Ordered container lemma for a hereditary relational property]
Let $\cP$ be a hereditary $\tau$-property.  For
$\varepsilon,\eta>0$ and $N,q\in\N$, we say that $\cP$ admits an
$(\varepsilon,\eta,N,q)$-ordered container lemma if, for every
$n\geq N$ and every $n$-vertex $\tau$-structure $\cA$ that is
$\varepsilon$-far from $\cP$, there are maps
\(f:\{U\subseteq V(\cA):\cA[U]\in\cP\} \longrightarrow\bigcup_{j=0}^{q}V(\cA)_{\neq}^{j}\)
and
\(C:\operatorname{im}(f) \longrightarrow\binom{V(\cA)}{\leq(1-\eta)n}\)
such that every entry of $f(U)$ belongs to $U$ and
$U\subseteq C(f(U))$.
\end{defn}

\begin{thm}[Container characterization for hereditary relational structures]
\label{thm:relational-hereditary-characterization}
Let $\cP$ be a hereditary property of finite $\tau$-structures.
\begin{enumerate}
\item If $\cP$ has a one-sided induced-sampling tester with sample
complexity $s(\varepsilon)$, then, for every $\varepsilon>0$, it admits an
ordered container lemma with
\(N(\varepsilon)=s(\varepsilon)\), \(q(\varepsilon)=s(\varepsilon)-1\), and \(\eta(\varepsilon)={2}/{(3s(\varepsilon))}\).

\item Conversely, suppose that $\cP$ admits ordered container lemmas
with parameter functions $\eta$, $N$, and $q$.  Then it has a one-sided
induced-sampling tester with sample complexity at most
\begin{equation}
\label{eq:relational-hereditary-sample}
T_{\cP}(\varepsilon)
=C\max\left\{
N(\varepsilon),
\frac{q(\varepsilon)+1}{\eta(\varepsilon)}
\log\left(
\frac{q(\varepsilon)+1}{\eta(\varepsilon)}+2
\right)
\right\},
\end{equation}
where $C$ is an absolute constant.
\end{enumerate}
\end{thm}

\begin{proof}
Fix $\varepsilon>0$ and write $s=s(\varepsilon)$.  Let $\cA$ be an
$n$-vertex structure, where $n\geq s$, that is $\varepsilon$-far from $\cP$.
By Lemma~\ref{lem:relational-membership-rule}, at least a $2/3$-fraction of
the $s$-subsets $S\subseteq V(\cA)$ satisfy $\cA[S]\notin\cP$.  Define an
$s$-uniform obstruction hypergraph $H_{\cA}$ on $V(\cA)$ by
\(E(H_{\cA}) =\{S\in\tbinom{V(\cA)}s:\cA[S]\notin\cP\}.\)
Every $U\subseteq V(\cA)$ with $\cA[U]\in\cP$ is independent in
$H_{\cA}$, by heredity.  Since
$e(H_{\cA})\geq(2/3)\binom{n}{s}$, Lemma~\ref{lem:dense-ordered-containers}
gives fingerprints of length at most $s-1$ and containers of size at most
$(1-2/(3s))n$.  This proves the first assertion.

Conversely, fix $\varepsilon>0$ and choose $T=T_{\cP}(\varepsilon)$ with the
constant large enough that Lemma~\ref{lem:counting-from-containers} applies.
If the input has fewer than $T$ vertices, inspect it completely.  Otherwise,
choose a uniformly random $T$-subset $S$, reveal $\cA[S]$, and reject exactly
when $\cA[S]\notin\cP$.  Heredity gives one-sidedness.  If $\cA$ is
$\varepsilon$-far and has at least $T\geq N(\varepsilon)$ vertices, apply the
assumed ordered container lemma to the family of sets $U$ with
$\cA[U]\in\cP$.  Lemma~\ref{lem:counting-from-containers} then shows that
\(\Pr[\cA[S]\in\cP]\leq\frac13.\)
Thus the tester rejects every $\varepsilon$-far input with probability at
least $2/3$.
\end{proof}

Combining this theorem with canonicalization gives the query-complexity form.

\begin{cor}[Relation-query formulation]
\label{cor:relational-query-characterization}
Let $\tau$ have maximum arity $r$, and let $\cP$ be hereditary.
\begin{enumerate}
\item If $\cP$ has a size-oblivious one-sided tester making at most
$Q(\varepsilon)$ relation queries, then, for constants
$c_\tau,C_\tau>0$, it admits ordered-container parameters satisfying
\(N(\varepsilon),q(\varepsilon) \leq C_\tau Q(c_\tau\varepsilon), \qquad \eta(\varepsilon) \geq\frac{1}{C_\tau Q(c_\tau\varepsilon)}.\)

\item Conversely, ordered-container parameters $(\eta,N,q)$ give a
size-oblivious one-sided tester with relation-query complexity
\[
O_\tau\!\left(
\max\left\{
N(\varepsilon)^r,
\left(
\frac{q(\varepsilon)+1}{\eta(\varepsilon)}
\log\left(
\frac{q(\varepsilon)+1}{\eta(\varepsilon)}+2
\right)
\right)^r
\right\}
\right).
\]
\end{enumerate}
\end{cor}

\begin{proof}
The first implication follows from
Theorem~\ref{thm:relational-canonicalization} and
Theorem~\ref{thm:relational-hereditary-characterization}(1).  For the reverse
implication, use the induced-sampling tester from
Theorem~\ref{thm:relational-hereditary-characterization}(2) and reveal every
relation on the sampled set.  The query bound follows from
\eqref{eq:relation-query-cost}.
\end{proof}

In particular, for every fixed bounded-arity signature, polynomial one-sided
induced-sampling complexity, polynomial relation-query complexity, and
polynomial ordered-container parameters are equivalent.

\subsection{Arbitrary properties and semi-hereditary envelopes}

The passage from hereditary to arbitrary properties is also independent of
the signature.

\begin{defn}[Semi-hereditary relational property]
A $\tau$-property $\cP$ is \textit{semi-hereditary} if there are a hereditary
property $\cH\supseteq\cP$ and a function $M:(0,1)\to\N$ such that, for every
$\varepsilon>0$, every $\cA\in\cH$ with
$|V(\cA)|\geq M(\varepsilon)$ is $\varepsilon$-close to $\cP$.  We call
$(\cH,M)$ a semi-hereditary witness.
\end{defn}

\begin{thm}[Container characterization for arbitrary relational properties]
\label{thm:relational-semi-characterization}
Let $\cP$ be a property of finite $\tau$-structures.
\begin{enumerate}
\item Suppose that $\cP$ has an oblivious one-sided induced-sampling
tester with sample complexity $s(\varepsilon)$.  Then $\cP$ has a
semi-hereditary witness $(\cH,M)$ satisfying
\(M(\varepsilon)\leq s(\varepsilon)\).
The same tester tests $\cH$, and $\cH$ admits ordered container parameters
\(N(\varepsilon)=s(\varepsilon)\), \( q(\varepsilon)=s(\varepsilon)-1\), and \(\eta(\varepsilon)={2}/{(3s(\varepsilon))}\).

\item Conversely, suppose that $\cP$ has a semi-hereditary witness
$(\cH,M)$ and that $\cH$ admits ordered container parameters
$(\eta,N,q)$.  Then $\cP$ has an oblivious one-sided induced-sampling
tester with sample complexity
\begin{equation}
\label{eq:relational-semi-sample}
O\!\left(
\max\left\{
M(\varepsilon/2),
N(\varepsilon/2),
\frac{q(\varepsilon/2)+1}{\eta(\varepsilon/2)}
\log\left(
\frac{q(\varepsilon/2)+1}{\eta(\varepsilon/2)}+2
\right)
\right\}
\right).
\end{equation}
\end{enumerate}
\end{thm}

\begin{proof}
Fix the given one-sided induced-sampling tester.  Let $\cB$ be the family of
all finite $\tau$-structures that can occur as a rejection pattern for this
tester at some proximity parameter, and let $\cH$ be the property of
containing no member of $\cB$ as an induced substructure.  Then $\cH$ is
hereditary.  One-sidedness implies $\cP\subseteq\cH$: otherwise, an input in
$\cP$ containing a rejection pattern would be rejected with positive
probability.

For each fixed $\varepsilon$, the original tester is one-sided for $\cH$,
since no induced substructure of an input in $\cH$ is a rejection pattern.  If
an input is $\varepsilon$-far from $\cH$, then it is also
$\varepsilon$-far from $\cP$, because $\cP\subseteq\cH$, and hence the tester
rejects with probability at least $2/3$.  Thus the same tester tests $\cH$.

We next verify semi-hereditariness.  Suppose that $\cA\in\cH$,
$|V(\cA)|\geq s(\varepsilon)$, and $\cA$ is $\varepsilon$-far from $\cP$.
The tester must reject with positive probability, so some sampled
$s(\varepsilon)$-vertex induced substructure belongs to $\cB$, contradicting
$\cA\in\cH$.  Hence one may take $M(\varepsilon)=s(\varepsilon)$.  Applying
Theorem~\ref{thm:relational-hereditary-characterization}(1) to $\cH$ gives the
container parameters in the first assertion.

Conversely, fix $\varepsilon>0$.  If the input has fewer than
$M(\varepsilon/2)$ vertices, inspect it completely.  Otherwise run the tester
for $\cH$ supplied by
Theorem~\ref{thm:relational-hereditary-characterization}(2), at proximity
parameter $\varepsilon/2$.  Inputs in $\cP$ are accepted because
$\cP\subseteq\cH$.  If $\cA$ is $\varepsilon$-far from $\cP$ and has at least
$M(\varepsilon/2)$ vertices, then it is $\varepsilon/2$-far from $\cH$.
Indeed, otherwise there would be a structure $\cA'\in\cH$ on the same vertex
set with $d_\tau(\cA,\cA')\leq\varepsilon/2$; semi-hereditariness would make
$\cA'$ $\varepsilon/2$-close to $\cP$, and the triangle inequality would
contradict the assumption on $\cA$.  Therefore the tester for $\cH$ rejects
with probability at least $2/3$, and the sample bound is
\eqref{eq:relational-semi-sample}.
\end{proof}

Together with Theorem~\ref{thm:relational-canonicalization} and
\eqref{eq:relation-query-cost}, this gives a query-complexity characterization
of arbitrary bounded-arity properties.  The quantitative information is
carried jointly by the ordered-container parameters of the hereditary envelope
and by the threshold $M$, which measures the difference between the original
property and that envelope.

\part{Applications}

The preceding parts identify ordered containers as the common quantitative structure behind one-sided testability.  We now use that structure in three directions.  The first two applications depend only on the binary nature of the ambient relations, which makes the edit cost of changing a small exceptional vertex set quadratic.  The third applies directly to every graph property covered by the characterization.
\begin{rem}
One may ask whether the applications below can be deduced directly from the container characterization. This is possible, but it is convenient to use a slightly stronger version of the characterization in which, in addition to containing every induced $\Pi$-subgraph, each container itself induces a graph that is close to $\Pi$. Such containers can be obtained by iterating the container construction until the auxiliary hypergraph of local obstructions becomes sparse. The resulting statement has somewhat weaker quantitative bounds, since the longer procedure requires larger fingerprints, but it makes the applications below essentially formal.
\end{rem}
\section{Partition properties}
\label{sec:partition-applications}

Given properties $\cP_1,\dots,\cP_k$ of structures in a common binary signature, we write $\Part(\cP_1,\dots,\cP_k)$ for the property of admitting a vertex partition whose $i$th part induces a member of $\cP_i$.  The next theorem gives the signature-independent closure principle.

\subsection{Binary relational structures}

For $\tau$-properties $\cP_1,\dots,\cP_k$, let
$\Part(\cP_1,\dots,\cP_k)$ consist of the structures $\cA$ admitting a
partition $V(\cA)=V_1\cup\cdots\cup V_k$ with
$\cA[V_i]\in\cP_i$ for every $i$.  Relation tuples meeting two different
parts are unrestricted.

\begin{thm}[Partition closure for binary signatures]
\label{thm:relational-partition-closure}
Let every symbol of $\tau$ have arity two, and let
$\cP_1,\dots,\cP_k$ be nonempty hereditary and extendable
$\tau$-properties.  Suppose that $\cP_i$ has one-sided induced-sampling
complexity $S_i(\delta)$, and put
$S(\delta)=\max_iS_i(\delta)$.  Then
$\Part(\cP_1,\dots,\cP_k)$ has a one-sided induced-sampling tester with sample
complexity
\(\widetilde O\!\left( \frac{k}{\varepsilon} S\!\left(\frac{\varepsilon}{8k}\right)^2 \right).\)
Consequently, its relation-query complexity is the square of this quantity,
up to a constant depending on $\tau$.
\end{thm}

\begin{proof}
Set
\(\gamma=\frac{\varepsilon}{8k}\),
\(\alpha=\sqrt{\frac{\varepsilon}{8k}}\),
\(\lambda=\frac{\varepsilon}{12}\), and \(s=S(\gamma)\). 

By increasing $s$ if necessary, we may assume that, for every $i$, every
structure that is $\gamma$-far from $\cP_i$ has at least a $2/3$-fraction of
bad induced $s$-substructures whenever its order is at least $s$.  Indeed, if a
larger sampled structure belonged to $\cP_i$, then all of its induced
substructures of the original sample size would belong to $\cP_i$ by
heredity.  Inputs below
any fixed threshold depending on $k,\varepsilon$, and $s$ will be inspected
completely, so throughout the structural argument we assume
$n\geq s/\alpha$.

Let $\cA$ be $\varepsilon$-far from
$\Part(\cP_1,\dots,\cP_k)$.  We first claim that whenever
$U_1,\dots,U_k\subseteq V(\cA)$ satisfy
$|\bigcup_iU_i|>(1-\lambda)n$, some $i$ satisfies
\[
|U_i|\geq\alpha n
\qquad\text{and}\qquad
\cA[U_i]\text{ is }\gamma\text{-far from }\cP_i.
\tag{\ref{thm:relational-partition-closure}.1}
\]
Suppose not.  Choose disjoint sets $V_i\subseteq U_i$ that partition
$U=\bigcup_iU_i$.  If $|U_i|\geq\alpha n$, edit $\cA[U_i]$ into a member of
$\cP_i$ and restrict the resulting structure to $V_i$; heredity shows that the
restriction still belongs to $\cP_i$, and the total normalized cost over all
such indices is at most $k\gamma=\varepsilon/8$.  If
$|U_i|<\alpha n$, replace the complete structure on $V_i$ by a member of
$\cP_i$ of the same order.  Such a member exists because $\cP_i$ is nonempty,
hereditary, and extendable, and the total cost is at most
$k\alpha^2=\varepsilon/8$.  Finally, assign the leftover set
$L=V(\cA)\setminus U$ to one part and use extendability.  Only relation tuples
meeting $L$ need be changed, at normalized cost at most
$2|L|/n+(|L|/n)^2<3\lambda=\varepsilon/4$.  The total cost is less than
$\varepsilon$, a contradiction.  This proves the claim.

For each $i$, let $H_i$ be the $s$-uniform hypergraph on $V(\cA)$ whose edges
are the $s$-sets inducing structures outside $\cP_i$.  Fix an arbitrary total
order of the vertex set and use the canonical maps furnished by the proof of
Lemma~\ref{lem:dense-ordered-containers}.  Let $W$ satisfy the partition
property and fix a witnessing partition
$W=W_1\cup\cdots\cup W_k$.  Initialize $C_i^{(0)}=V(\cA)$ for every $i$.
While the union of the current containers has size greater than
$(1-\lambda)n$, the claim gives an index $i$ for which
$|C_i|\geq\alpha n$ and $\cA[C_i]$ is $\gamma$-far from $\cP_i$.  Choose the
least such index.  Then $H_i[C_i]$ has density at least $2/3$, and $W_i$ is an
independent set in this hypergraph.  Applying
Lemma~\ref{lem:dense-ordered-containers} inside $C_i$ replaces $C_i$ by a set
containing $W_i$ and of size at most
\((1-2/(3s))|C_i|\), while recording an ordered block of at most $s-1$
vertices of $W_i$.

An index can be selected at most
\(R_0= 1+\left\lceil \frac{\log(1/\alpha)}{-\log(1-2/(3s))} \right\rceil =O\!\left(s\log\frac1\alpha\right)\)
times.  Thus the procedure has at most $R=kR_0$ rounds and records at most
\(q=R(s-1) =O\!\left(ks^2\log\frac1\alpha\right)\)
vertex occurrences.  Its terminal union $C$ contains $W$ and satisfies
$|C|\leq(1-\lambda)n$.

The complete run is determined by its at most $q$ recorded vertex
occurrences, the round boundaries, and the selected indices.  After identifying
repeated vertices, the number of possible auxiliary transcripts is
$\exp(\widetilde O(q))$.  Summing over these transcripts and over the ordered
fingerprints gives, exactly as in the standard container-counting argument, that
a uniformly random $t$-set satisfies the partition property with probability at
most $1/3$ for
\(t=\widetilde O\!\left(\frac{q+1}{\lambda}\right)
 =\widetilde O\!\left(\frac{k}{\varepsilon}s^2\right).\)
We take $t$ also larger than the fixed small-input threshold; this does not
change the displayed bound.  The tester samples $t$ vertices, inspects the
whole induced structure, and rejects exactly when it fails the partition
property.  It is one-sided, and the preceding bound gives soundness.  The
relation-query bound follows from \eqref{eq:relation-query-cost} with $r=2$.
\end{proof}

\subsection{Graphs and generalized colouring parameters}

Specializing the preceding theorem to the signature of simple graphs and using canonicalization gives the following query-complexity formulation.

\begin{thm}[Testing partitions of hereditary graph properties]
\label{thm:partitioned-hereditary-properties}
Let $k\geq1$, and let $\Pi_1,\dots,\Pi_k$ be hereditary and extendable graph properties.  Suppose that $\Pi_i$ admits a size-oblivious one-sided error tester of query complexity at most $Q_i(\varepsilon)$, and let
\(Q(\varepsilon)=\max_{i\in[k]}Q_i(\varepsilon)\).
Then $\Part(\Pi_1,\dots,\Pi_k)$ admits a size-oblivious one-sided error tester with query complexity
\(\widetilde O\!\left( \frac{k^2}{\varepsilon^2} Q^4\!\left(\frac{\varepsilon}{Ck}\right) \right),\)
where $C>0$ is an absolute constant.
\end{thm}

\begin{proof}
By canonicalization, each $\Pi_i$ has one-sided induced-sampling complexity $S_i(\delta)=O(Q_i(c\delta))$ for an absolute constant $c>0$.  Apply Theorem~\ref{thm:relational-partition-closure} to the symmetric irreflexive binary signature.  The resulting tester samples
\(\widetilde O\!\left( \frac{k}{\varepsilon} Q^2\!\left(\frac{\varepsilon}{Ck}\right) \right)\)
vertices and queries every pair among them.  Squaring the sample bound proves the claim.
\end{proof}

For integers $r,s\geq0$ with $r+s\geq1$, a graph is
\textit{$(r,s)$-colourable} if its vertex set can be partitioned into $r$
independent sets and $s$ cliques; empty parts are allowed.  We denote this
property by $\Pi_{r,s}$. Thus $\Pi_{k,0}$ is ordinary $k$-colourability.  A \textit{split graph} is a
$(1,1)$-colourable graph, equivalently a graph whose vertex set is the union of
one independent set and one clique.  The \textit{cochromatic number} $z(G)$ is
the least $k$ for which $V(G)$ can be partitioned into $k$ parts, each inducing
either an independent set or a clique.  Equivalently,
\(z(G)=\min\{r+s:G\in\Pi_{r,s}\}\).

\begin{cor}[$(r,s)$-colourability]
\label{cor:rs-colourability}
For fixed $r,s\geq0$ with $r+s\geq1$, the property $\Pi_{r,s}$ admits a size-oblivious one-sided error tester with query complexity
\(\widetilde O\!\left(
{(r+s)^6}/{\varepsilon^6}
\right)\).
\end{cor}

\begin{proof}
The properties of being edgeless and complete are hereditary and extendable, and each has a one-sided tester using $O(1/\varepsilon)$ queries.  Apply Theorem~\ref{thm:partitioned-hereditary-properties} with $k=r+s$.
\end{proof}

Two familiar cases are worth recording separately.

\begin{cor}\label{cor:color}
For every $k\geq1$, $k$-colourability admits a size-oblivious one-sided error tester with query complexity $\widetilde O(k^6/\varepsilon^6)$.
\end{cor}

\begin{cor}\label{cor:split}
The property of being a split graph admits a size-oblivious one-sided error tester with query complexity $\widetilde O(1/\varepsilon^6)$.
\end{cor}

For a digraph $D$, the \textit{dichromatic number} $\vec\chi(D)$ is the least
integer $k$ for which $V(D)$ can be partitioned into $k$ sets, each inducing an
acyclic digraph.

\begin{cor}[Dichromatic number]\label{cor:dichromatic-number}
For every $k\geq1$, the property $\vec\chi(D)\leq k$ admits a one-sided
induced-sampling tester with sample complexity
\(\widetilde O\!\left(\frac{k^3}{\varepsilon^3}\right).\)
Consequently, it admits a size-oblivious one-sided-error tester with
relation-query complexity
\(\widetilde O\!\left(\frac{k^6}{\varepsilon^6}\right).\)
\end{cor}

\begin{proof}
Let $\cP_{\mathrm{acyc}}$ be the property of being acyclic.  A digraph has
dichromatic number at most $k$ exactly when it belongs to
$\Part(\cP_{\mathrm{acyc}},\ldots,\cP_{\mathrm{acyc}})$ with $k$ factors.
Bender and Ron \cite{bender2002testing} prove that a uniformly random set of
$\widetilde O(1/\delta)$ vertices detects a directed cycle with probability at
least $2/3$ in every digraph that is $\delta$-far from acyclic.  Thus
$S_{\mathrm{acyc}}(\delta)=\widetilde O(1/\delta)$.  Applying
Theorem~\ref{thm:relational-partition-closure} gives sample complexity
\[
\widetilde O\!\left(
\frac{k}{\varepsilon}
S_{\mathrm{acyc}}\!\left(\frac{\varepsilon}{8k}\right)^2
\right)
=\widetilde O\!\left(\frac{k^3}{\varepsilon^3}\right).
\]
Revealing all ordered pairs in the sample gives the stated relation-query
bound.
\end{proof}

\begin{rem}
The sample bound in the Bender--Ron tester is important here.  Their
adjacency-matrix query bound is $\widetilde O(\delta^{-2})$; applying the graph
query-complexity closure theorem directly to that estimate would give only
$\widetilde O(k^{10}\varepsilon^{-10})$.
\end{rem}

\begin{cor}[Cochromatic number]
\label{cor:cochromatic-number}
For every $k\geq1$, the property of having cochromatic number at most $k$ admits a size-oblivious one-sided error tester with query complexity
\(\widetilde O\!\left({k^7}/{\varepsilon^6}\right).
\)
\end{cor}

\begin{proof}
Because empty parts are allowed, $z(G)\leq k$ if and only if
$G\in\bigcup_{r=0}^{k}\Pi_{r,k-r}$.  Run independent amplified one-sided
testers for these $k+1$ properties and reject precisely when all of them
reject.  Completeness is one-sided.  If the input is $\varepsilon$-far from the
union, then it is $\varepsilon$-far from every constituent property; amplifying
each tester to error probability at most $1/(3(k+1))$ and applying the union
bound gives soundness at least $2/3$.  The factor $k+1$ from running all the
testers gives the stated bound, with amplification absorbed by the
polylogarithmic factor.
\end{proof}

\section{Linearly large induced substructures}
\label{sec:large-induced-applications}

The next application concerns properties asserting that there is a linearly large induced substructure satisfying $\cP$.  Such properties are generally not hereditary and need not be expressible through a bounded vertex partition.  Nevertheless, iterating the ordered-container construction forces every induced member of $\cP$ into a container whose density lies below the target density.

\subsection{A general large-induced-substructure theorem}

For $\rho\in(0,1]$ and a property $\cP$, write
\(\mathrm{Large}_\rho(\cP) =\{\cA:\text{there is }U\subseteq V(\cA),\ |U|\geq\rho|V(\cA)|, \ \cA[U]\in\cP\}.\)

\begin{thm}[Testing for a large induced member]
\label{thm:relational-large-induced}
Let $\tau$ be a finite binary signature, let $\cP$ be a nonempty hereditary
and extendable $\tau$-property, and let $S_{\cP}(\delta)$ be its one-sided
induced-sampling complexity.  Fix $\rho\in(0,1]$ and
$0<\varepsilon\leq\rho^2/8$, and let
\(\delta={\varepsilon}/{4\rho^2}\), \(\gamma={\varepsilon}/{16\rho}\), and \(   s=S_{\cP}(\delta)\).
Then $\mathrm{Large}_\rho(\cP)$ has a two-sided induced-sampling tester using
\(\widetilde O\!\left( \frac{\rho^2}{\varepsilon^2} s^2\log\frac2\rho \right)\)
vertices.  Its relation-query complexity is the square of this quantity, up
to a constant depending on $\tau$.
\end{thm}

\begin{proof}
Put $n=|V(\cA)|$.  We first establish the local-farness statement used by the
container iteration.  There is a threshold
\(n_0=O_\tau\!\left(\frac{s}{\rho}+\frac1\varepsilon+1\right),\)
such that if $n\geq n_0$ and $\cA$ is $\varepsilon$-far from
$\mathrm{Large}_\rho(\cP)$, then every
$U\subseteq V(\cA)$ with $|U|\geq(\rho-\gamma)n$ is $\delta$-far from
$\cP$.

Indeed, suppose that $\cA[U]$ can be changed into a structure
$\cB\in\cP$ at cost at most $\delta$ relative to $|U|^2$.  If
$|U|\geq\rho n$, choose a uniformly random set
$W\subseteq U$ of size $\lceil\rho n\rceil$.  Restricting the edit set to
$W$ and averaging shows that some such $W$ requires at most
\(\delta\rho^2n^2+O_\tau(n)\)
relation changes; the $O_\tau(n)$ term accounts for tuples with repeated
coordinates.  Since $\cB[W]\in\cP$, this makes $\cA$ closer than
$\varepsilon$ to $\mathrm{Large}_\rho(\cP)$ for sufficiently large $n$.
If instead $|U|<\rho n$, extend $\cB$ to a member of $\cP$ on
$\lceil\rho n\rceil$ vertices.  Besides the edits inside $U$, only binary
relation tuples meeting at most $\gamma n+1$ new vertices need be changed, at
normalized cost at most $3\gamma\rho+o(1)$.  Since
\(\delta\rho^2+3\gamma\rho =\frac{\varepsilon}{4}+\frac{3\varepsilon}{16} <\varepsilon,\)
we again obtain a contradiction.  Increase $n_0$ also so that
$(\rho-\gamma)n_0\geq s$.

Let $H$ be the $s$-uniform obstruction hypergraph on $V(\cA)$,
\(E(H)=\{S\in\tbinom{V(\cA)}s:\cA[S]\notin\cP\}.\)
For every $C\subseteq V(\cA)$ with $|C|>(\rho-\gamma)n$, the local-farness
statement and the definition of $s$ imply that $H[C]$ has density at least
$2/3$.  Fix a total order of the vertices and use the canonical maps from
Lemma~\ref{lem:dense-ordered-containers}.  Given a set $I$ with
$\cA[I]\in\cP$, start with $C_0=V(\cA)$ and, while
$|C_a|>(\rho-\gamma)n$, apply the dense ordered-container lemma to the
independent set $I\subseteq C_a$ in $H[C_a]$.  Each round records at most
$s-1$ vertices of $I$ and replaces the current container by one of size at
most $(1-2/(3s))|C_a|$.

The number of rounds is at most
\(R= 1+\left\lceil \frac{\log(1/(\rho-\gamma))}{-\log(1-2/(3s))} \right\rceil =O\!\left(s\log\frac2\rho\right),\)
where we used $\gamma\leq\rho/128$.  Thus the concatenated transcript has
length at most
\(q=R(s-1)=O\!\left(s^2\log\frac2\rho\right).\)
The complete run is determined by its at most $q$ recorded vertex
occurrences and its round boundaries.  Identifying repeated vertices leaves
$\exp(\widetilde O(q))$ possible auxiliary transcripts, each of which, together
with an ordered fingerprint of length at most $q$, determines a container of
size at most $(\rho-\gamma)n$.  A union bound over these choices, followed by
the hypergeometric tail bound, shows that it is enough to take
\(t=\widetilde O\!\left(\frac{q+1}{\gamma^2}\right)
 =\widetilde O\!\left(\frac{\rho^2}{\varepsilon^2}s^2\log\frac2\rho\right).\)
Choose $t$ also larger than $n_0$.
The tester samples a uniformly random $t$-set $T$ and accepts precisely when
there is $I\subseteq T$ such that
\(|I|\geq(\rho-\gamma/2)t \qquad\text{and}\qquad \cA[I]\in\cP.\)
If the input has fewer than $t$ vertices, it is inspected completely.

If $\cA$ contains an induced member of $\cP$ on at least $\rho n$ vertices,
Hoeffding's inequality for the hypergeometric distribution gives
\(\Pr\bigl[|T\cap I|<(\rho-\gamma/2)t\bigr] \leq e^{-\gamma^2t/2}<\frac13\)
for a sufficiently large implicit constant; heredity then gives completeness.
If $\cA$ is $\varepsilon$-far, the preceding union bound shows that the
acceptance probability is at most $1/3$.  This proves the sample bound.  The
relation-query bound follows from \eqref{eq:relation-query-cost} with $r=2$.
\end{proof}

\subsection{Large acyclic subdigraphs}

For $\rho\in(0,1]$, we call a digraph $D$ on $n$ vertices a
$\rho$-DAG if it contains an induced acyclic subdigraph on at least $\rho n$
vertices.  Since acyclicity is monotone under deletion of arcs, $D$ is
$\varepsilon$-far from this property exactly when, for every
$S\subseteq V(D)$ with $|S|=\lceil\rho n\rceil$, more than
$\varepsilon n^2$ arcs must be deleted from $D[S]$ to make it acyclic.

\begin{thm}\label{thm:rho-DAG}
Let $\rho\in(0,1]$.  The property of being a $\rho$-DAG admits a size-oblivious two-sided error tester with query complexity
\(\widetilde O\!\left(\rho^{12}\varepsilon^{-8}\right)
\)
whenever $0<\varepsilon\leq\rho^2/8$.  For arbitrary $\varepsilon>0$, the query complexity is
\(\widetilde O\!\left(
\rho^{12}\min\{\varepsilon,\rho^2\}^{-8}
\right)\).
\end{thm}

\begin{proof}
Apply Theorem~\ref{thm:relational-large-induced} to the binary directed signature and to the hereditary, extendable property of acyclicity.  Bender and Ron \cite{bender2002testing} show that a uniformly random set of
$O(\delta^{-1}\log(1/\delta))$ vertices detects a directed cycle with
probability at least $2/3$ in every digraph that is $\delta$-far from acyclic.
Thus
\(S_{\mathrm{acyc}}(\delta)=\widetilde O(1/\delta).\)
With $\delta=\varepsilon/(4\rho^2)$, Theorem~\ref{thm:relational-large-induced} gives sample complexity $\widetilde O(\rho^6\varepsilon^{-4})$.  Revealing all ordered pairs in the sample gives the asserted query bound.  The second statement follows by monotonicity in $\varepsilon$.
\end{proof}

\begin{cor}[Polynomial preservation]
\label{cor:large-induced-polynomial}
Let $\tau$ be a fixed binary signature and let $\cP$ be hereditary and extendable.  If $\cP$ has polynomial one-sided induced-sampling complexity, then $\mathrm{Large}_\rho(\cP)$ has polynomial two-sided induced-sampling complexity for every fixed $\rho>0$.
\end{cor}

\begin{proof}
This is immediate from Theorem~\ref{thm:relational-large-induced}.
\end{proof}

\section{Counting and structure inside far hosts}
\label{sec:far-host-applications}

The characterization contains information beyond testing.  A short fingerprint and a smaller container give an immediate entropy deficit for the family of all induced substructures satisfying the property.

\begin{cor}[Counting induced members]
\label{cor:counting-induced-members}
Let $\Pi$ be a hereditary graph property, and suppose that for an $n$-vertex graph $G$ every set $U\subseteq V(G)$ with $G[U]\in\Pi$ has an ordered fingerprint of length at most $q$ whose associated container has size at most $(1-\eta)n$.  Then
\(\bigl|\{U\subseteq V(G):G[U]\in\Pi\}\bigr| \leq (q+1)n^q 2^{(1-\eta)n}.\)
Moreover, for every $t\geq q$,
\(\frac{ |\{U\in\binom{V(G)}t:G[U]\in\Pi\}| }{\binom{n}{t}} \leq \sum_{j=0}^{q}t^j(1-\eta)^{t-j}.\)
\end{cor}

\begin{proof}
There are at most $\sum_{j=0}^{q}n^j\leq(q+1)n^q$ ordered fingerprints.  Every induced member assigned to a fixed fingerprint is a subset of its container, which has at most $(1-\eta)n$ vertices.  This proves the first bound.  The second is the counting argument of Lemma~\ref{lem:counting-from-containers}, specialized to the obstruction hypergraph associated with $G$ and $\Pi$.
\end{proof}

\subsection{Random hosts and edit distance}

The study of hereditary graph properties has largely focused on the structure
and typical form of graphs belonging to the property; see, for example,
\cite{alon2011structure}
and the references therein.  Here the host graph lies far outside the
property, and the question is how the induced members of the property can be
distributed inside it.

For a graph property $\Pi$, let
\(d_\Pi(G)=\min\{|E(G)\triangle E(G')|:G'\in\Pi, \,V(G')=V(G)\}\)
and let $\mathrm{ed}(n,\Pi)=\max_{|V(G)|=n}d_\Pi(G)$.  Alon and Stav
\cite{alon2008furthest} proved that for every hereditary property $\Pi$ there
is a density $p=p(\Pi)$ such that, with high probability,
\(\mathrm{ed}(n,\Pi)=d_\Pi(G(n,p))+o(n^2).\)
Their subsequent work \cite{alon2008maximum} determines the extremal density
and the limiting distance for several natural classes and, in particular,
determines the distance of $G(n,1/2)$ from every hereditary property.

Our characterization refines this edit-distance statement at the level of all
induced members of $\Pi$.  Fix $\varepsilon>0$.  If $\Pi$ has a one-sided
tester of query complexity $Q$, then every sufficiently large graph that is
$\varepsilon$-far from $\Pi$ has ordered-container parameters
$q=O(Q(c\varepsilon))$ and
$\eta=\Omega(Q(c\varepsilon)^{-1})$.  Hence every induced $\Pi$-subgraph lies
in a set of size at most $(1-\eta)n$, and that set is determined by an ordered
tuple of at most $q$ vertices.  In particular, the possible locations of all
induced members are controlled by only polynomially many containers.  The
following random-host consequence isolates precisely what is needed later.

\begin{cor}\label{cor:random-hereditary}
Let $\Pi$ be a hereditary graph property, and suppose that there exist $p\in[0,1]$ and $\varepsilon_0>0$ such that $G(n,p)$ is $\varepsilon_0$-far from $\Pi$  with probability that tends to $1$ as $n\rightarrow\infty$. Then there exist constants $\eta_0>0$ and  $q_0\in\N$, such that, with probability that tends to $1$ as $n\rightarrow\infty$, every set $U\subseteq V(G(n,p))$ with $G(n,p)[U]\in\Pi$ is contained in one of at most $n^{q_0}$ vertex sets of size at most $(1-\eta_0)n$.
\end{cor}
\begin{proof}
By the hereditary testing theorem of Alon and Shapira, $\Pi$ has a
size-oblivious one-sided tester.  Apply
Theorem~\ref{thm:hereditary-characterization} at the fixed proximity parameter
$\varepsilon_0$.  This gives constants $\eta>0$, $q\in\N$, and $N\in\N$ such
that every $n$-vertex graph with $n\geq N$ that is $\varepsilon_0$-far from $\Pi$ has
an ordered-container lemma with fingerprint length at most $q$ and container
size at most $(1-\eta)n$.

With high probability, the realization of $G(n,p)$ is
$\varepsilon_0$-far from $\Pi$, so the preceding conclusion applies.  There
are at most $\sum_{j=0}^{q}(n)_j\leq(q+1)n^q$ possible ordered fingerprints.
Set $q_0=q+1$.  After increasing the fixed threshold $N$ if necessary,
$(q+1)n^q\leq n^{q_0}$.  Taking $\eta_0=\eta$, the corresponding containers
have the required size and cover every set inducing a member of $\Pi$.
\end{proof}

To obtain a concrete example, we use the exact asymptotic distance of
$G(n,1/2)$ from a hereditary property established by Alon and Stav
\cite{alon2008maximum}.  Define the \textit{binary chromatic number} of $\Pi$
by
\(\chi_B(\Pi) = 1+\max\{r+s:\Pi_{r,s}\subseteq \Pi\},\)
where $\Pi_{r,s}$ denotes the property of being $(r,s)$-colorable. 

\begin{thm}[Alon--Stav \cite{alon2008maximum}]\label{thm:alon-stav-half}
Let $\Pi$ be a hereditary graph property for which
$2\leq\chi_B(\Pi)<\infty$.  Then, with high probability,
\(d_\Pi(G(n,1/2)) = \left( \frac{1}{2(\chi_B(\Pi)-1)}\pm o(1) \right)\binom{n}{2}.\)
\end{thm}

Under the hypotheses of the theorem, this quantity is a positive constant multiple of
$n^2$, so Corollary~\ref{cor:random-hereditary} applies to $p=1/2$.

Let $\Pi_{\mathrm{perf}}$ be the property of being perfect.  Every bipartite
graph is perfect, so $\Pi_{2,0}\subseteq\Pi_{\mathrm{perf}}$ and hence
$\chi_B(\Pi_{\mathrm{perf}})\geq3$.  Conversely, $C_5$ is not perfect and is
$(r,s)$-colourable for every $r,s\geq0$ with $r+s=3$.  Therefore no
$\Pi_{r,s}$ with $r+s=3$ is contained in $\Pi_{\mathrm{perf}}$, and
$\chi_B(\Pi_{\mathrm{perf}})=3$.  The Alon--Stav theorem gives, with high
probability,
\(d_{\Pi_{\mathrm{perf}}}(G(n,1/2)) =\left(\frac14\pm o(1)\right)\binom n2.\)
Combining this with Corollary~\ref{cor:random-hereditary} yields the following.

\begin{cor}\label{cor:perfectness}
There exist constants $\eta>0$ and $q\in\N$ such that, with probability
tending to $1$ as $n\to\infty$, every perfect induced subgraph of
$G(n,1/2)$ is contained in one of at most $n^q$ vertex sets of size at most
$(1-\eta)n$.
\end{cor}

\printbibliography

\end{document}